\documentclass[11pt]{amsart}
\usepackage[margin=1.02in]{geometry}
\usepackage{amsmath}
\usepackage{amsfonts}
\usepackage{amssymb}
\usepackage{amsthm}
\usepackage{hyperref}
\usepackage{tikz}
\input xy 
\xyoption{all}
\usepackage{enumerate}
\usepackage{cleveref}
\usepackage{mathrsfs}
\hypersetup{colorlinks=true,linkcolor=blue,citecolor=magenta}
\newtheorem{theorem}{Theorem}
\newtheorem*{theorem*}{Theorem}

\newtheorem{proposition}[theorem]{Proposition}

\newtheorem{prediction}{Prediction}
\newtheorem{assumption}{Simplifying Assumption}
\newtheorem{definition}[theorem]{Definition}
\newtheorem{remark}{Remark}
\newtheorem*{result}{Result of computation}
\newtheorem{conjecture}[theorem]{Conjecture}
\Crefname{conjecture}{Conjecture}{Conjectures}

\theoremstyle{plain}

\theoremstyle{plain}

\numberwithin{equation}{section}
\numberwithin{theorem}{section}
\numberwithin{remark}{section}

\author[Botkin, Dawsey, Hemmer, Just and Schneider]{Aidan Botkin, Madeline L. Dawsey, David J. Hemmer,\\ Matthew R. Just and Robert Schneider}

\address{Department of Mathematical Sciences\newline
Michigan Technological University\newline
Houghton, Michigan 49931, U.S.A.}
\email{abotkin@mtu.edu}

\address{Department of Mathematics\newline
University of Texas at Tyler\newline Tyler, Texas 75799, U.S.A.}
\email{mdawsey@uttyler.edu}

\address{Department of Mathematical Sciences\newline
Michigan Technological University\newline
Houghton, Michigan 49931, U.S.A.}
\email{djhemmer@mtu.edu}

\address{Department of Mathematics\newline University of Georgia\newline
Athens, Georgia 30602, U.S.A.}
\email{justmatt@uga.edu}

\address{Department of Mathematical Sciences\newline
Michigan Technological University\newline
Houghton, Michigan 49931, U.S.A.}
\email{robertsc@mtu.edu}

\title{Partition-theoretic model of prime distribution}
\begin{document}

\begin{abstract}
We make an application of ideas from partition  theory to a problem in multiplicative number theory. We propose a deterministic model of prime number distribution, from first principles related to properties of integer partitions,  that naturally predicts the prime number theorem as well as the twin prime conjecture.  The model posits that, for $n\geq 2$,     $$p_{n}\  =\  1\  +\  2\sum_{j=1}^{n-1}\left\lceil \frac{d(j)}{2}\right\rceil\  +\   \varepsilon(n),$$
 where $p_k$ is the $k$th prime number, $d(k)$ is the divisor function, and $\varepsilon(k)$ is an explicit error term that is negligible asymptotically; both the main term and error term   represent enumerative functions in our conceptual model. We refine the error term to give  numerical estimates of $\pi(n)$    similar to those provided by the logarithmic integral, and much more accurate than $\operatorname{li}(n)$ up to  $n=10{,}000$ where the estimates are {\it almost exact}. 
 We then perform  computational tests of     unusual predictions of the model, finding  limited  evidence of predictable variations in prime gaps. 
\end{abstract}

\maketitle
\thispagestyle{empty}

%

\section {Introduction}\label{sect1}

\subsection{Model of primes  from partition first principles}\label{sect1.1} 
The prime numbers mimic a random sequence in key ways, and many  successful models of prime distribution  are based upon probabilistic hypotheses \cite{Tenenbaum}. 
In his lecture notes \cite{Tao}, Terence Tao   
speaks to this point: 
\begin{quote}
{``We do  have a number of extremely convincing and well supported [random] models for the primes (and related objects) that let us predict what the answer to many prime number theory questions (both multiplicative and non-multiplicative) should be....    Indeed, the models are so effective at this task that analytic number theory is in the curious position of being able to confidently predict the answer to a large proportion of the open problems in the subject, whilst not possessing a clear way forward to rigorously confirm these answers!''}
\end{quote}

 In this paper, a sequel to \cite{supernorm}, we formulate a {\it deterministic} model of prime number distribution based on information about the integers   gleaned from facts about {integer partitions}.\footnote{Early versions of this model were  previously presented   by some of the authors in \cite{MAAIM} and the appendix of \cite{supernorm1}.}    
In short, the comparison of two multiplicative statistics defined on  partitions, the norm \cite{SS-norm}  and supernorm \cite{supernorm}, imposes constraints on the number line  unrelated to integer factorizations, which we leverage to give an estimate of the $n$th prime gap. Computations show the model to be reasonably accurate numerically.  We justify the model from partition-theoretic first principles, and we refine the error to derive an explicit  prediction for the value of the $n$th prime  number  which is surprisingly close at small numbers, yielding an estimate   for $\pi(n)$ (the number of primes less than or equal to $n$) that is {\it almost exact} up to   $n=10{,}000$; which predicts the prime number theorem  asymptotically as $n$ increases; and which   predicts the twin prime conjecture for reasons that would be difficult to justify without ideas about partitions. The model also makes  unusual predictions about local behaviors of prime gaps; we find limited computational evidence for one of these predictions. 
Finally, we note deficiencies in   the model and suggest   refinements. 

Below is the central  claim of our   model, which we formulate from first principles about partitions. For $x\in \mathbb R$, let $\left\lceil x\right\rceil$ denote the usual  {\it ceiling function}.  For $x\geq 0$, let $\left\lfloor x\right\rfloor$ denote the  {\it floor function};  throughout this paper {\it we modify the usual floor function definition to let  $\left\lfloor x\right\rfloor:=0$ if $x<0$}.\\

\noindent {\bf Partition model of prime numbers.} {\it The $n$th prime  number   $p_n,\  n\geq 1,$  is modeled by setting $p_1=2$ and for $n\geq 2$ by using the formula 
$$p_{n}\  =\  1\  +\  2\sum_{j=1}^{n-1}\left\lceil \frac{d(j)}{2}\right\rceil\  +\   \varepsilon(n),$$
where  $d(k)$ is the divisor function and $\varepsilon(k)$ is an explicit error   term  that is negligible by comparison. We refer to the case $\varepsilon(n):=0$ for all $n\geq 2$ as Model 1. More nuanced consideration of the error in Model 1 leads to a refined   model that we call   Model 2, which for $n\geq 2$ is the case 
$$\varepsilon(n):=\left\lfloor \pi_2(p_{n-1})-2\gamma(n-1) \right\rfloor,$$
where $\pi_2(k)$ is the number of semiprimes less than or equal to $k\geq 1$, and $\gamma=0.5772\dots$ is the Euler--Mascheroni constant. To simplify calculations, we set $p_1=2,\  p_2=3,$ and for $n\geq 3$ use an asymptotic approximation to the above error term, 
$$\varepsilon(n):=\left\lfloor (n-1)\cdot\left(\log \log(n-1)-2\gamma \right)\right\rfloor,$$
where $\log x$ denotes the natural logarithm, to yield a computational model that we call Model 2*.}\\

The reader is referred to Table \ref{table4} that gives a comparison of our models' predictions for   $\pi(n)$, alongside  corresponding predictions from the prime number theorem.\footnote{The  authors are grateful to Eli DeWitt (Michigan Technological Univ.) and Alexander Walker (Univ. of Georgia) for producing Python code and  computational data that advanced this project, as well as   Maxwell Schneider (Univ. of Georgia) for   consultation about  programming  and combinatorial algorithms, as  undergraduate research  students.} 

\vskip.25in

\begin{table}[h]    \centering
    \begin{tabular}{| c | c | c | c | c | c | c | }
\hline
 $n$\  & $\pi(n)$ & $n/\operatorname{log} n$ & $\operatorname{li}(n)$ & Model 1 &  Model 2 &   Model 2*   \\ \hline
 $10$\  & $4$ & $4.34...$ & $6.16...$ & 4  & 4 & 4 \\ \hline
 $100$\  & $25$ & $21.71...$ & $30.12...$  & 27 & 26 & 27  \\ \hline
  $1000$\  & $168$ & $144.76...$ & $177.60...$ & 184 & 168 & 171 \\ \hline
   $10{,}000$\  & $1229$ & $1085.73...$ & $1246.13...$ & 1352 & 1212 & 1233  \\ \hline
    $100{,}000$\  & $9592$ & $8685.88...$ & $9629.80...$ & 10{,}602 & 9435 & 9618  \\ \hline
    $1{,}000{,}000$\  & $78{,}498$ & $72{,}382.41...$ & $78{,}627.54...$ & 86{,}739 & 77{,}322  & 78{,}740 \\ \hline
\end{tabular}
\vskip.1in
    \caption{Comparing   estimates for $\pi(n)$}
    \label{table4}
    \end{table}

\subsection
{Concepts and notations}
Let $\mathbb Z^+$ denote the   {\it natural numbers}. Let  $\mathbb P$ denote the  {\it prime numbers}. 
Let $p_i\in\mathbb P$   denote the $i$th prime number, viz. $p_1=2, p_2=3, p_3=5,$ etc.; we define $p_0:=1$, and we refer to the subscript $i\in\mathbb Z^+$ of  $p_i\in\mathbb P$ as the {\it index} of the prime number.  
For $n\in \mathbb Z^+$,  we    write its {\it prime factorization} as $n=p_1^{a_1}p_2^{a_2}p_3^{a_3}\cdots p_i^{a_i}\cdots$,  $a_i\in \mathbb Z_{\geq 0}$,  where only finitely many primes $p_i$ have nonzero {\it multiplicity} (number of occurrences); we omit the factor $p_i$ from the notation if $a_i=0$. For $n\in \mathbb Z^+$, let $\pi(n)$ denote the {\it number of primes less than or equal to $n$}.


Let $\mathcal P$ denote the set of {\it integer partitions}, unordered finite multisets of natural numbers including the empty partition $\emptyset\in \mathcal P$ (see \cite{And}).   
For a nonempty partition $\lambda\in \mathcal P$, we notate $\lambda=(\lambda_1, \lambda_2, \lambda_3, \dots, \lambda_r)$,  with   {\it parts} $\lambda_i\in \mathbb Z^+$   written in weakly decreasing order $\lambda_1\geq \lambda_2\geq \dots \geq \lambda_r \geq 1$. 
For $\lambda \in \mathcal P$, let $|\lambda|\geq 0$ denote the {\it size} (sum of parts), let $\ell(\lambda)=r$ be the {\it length} (number of parts), 
and let $m_i=m_i(\lambda)\geq 0$ be the {\it multiplicity}   of $i\in \mathbb Z^+$ as a part of partition $\lambda$. Similarly to prime factorizations, we also   write partitions in {\it part-multiplicity} notation 
$\lambda = \left<1^{m_1} 2^{m_2} 3^{m_3}\dots i^{m_i}\dots \right>$, omitting any part $i\geq 1$ if $m_i=0$.  Let $p(n)$  be the {\it number of   partitions of size $n\geq 0$}. 

There are further  important {\it partition statistics}, i.e.,  functions $f\colon \mathcal P \to \mathbb Z$ such as the partition rank,   crank and others,  that encode   combinatorial properties in their values and often enjoy   modular congruence relations and other nice behaviors \cite{And}. Our model of prime distribution arises from consideration of two newly-defined, 
{\it multiplicative} partition statistics.

\subsection{Multiplicative partition statistics and a multiplicative theory of (additive) partitions}
In 
\cite{Robert_zeta, Robert_arithmetic, Robert_PhD}  and a  series of subsequent publications, the fifth author (Schneider)  introduces a {\it multiplicative} theory of integer partitions   that parallels multiplicative number theory in many respects\footnote{Inspired by Alladi-Erd\H{o}s  \cite{Alladi}, Andrews \cite{Andrews_ideal}, Granville \cite{Granville}, Granville-Soundararajan  \cite{G-S}, Ono \cite{Ono_web} and Zagier \cite{Zagier}.} and studies a variety of new, multiplicative partition statistics -- many of them representing partition-theoretic analogues of well-known arithmetic functions like the M\"{o}bius function $\mu(n)$, the Euler phi function $\varphi(n)$, and other  functions from multiplicative number theory \cite{Apostol1}.  In these works, the fifth author presents (and proves   cases of) the following philosophy: \vskip.1in
\noindent
 {\it Theorems in multiplicative number theory are special cases of more general theorems in partition theory. Theorems in $\mathbb Z^+$ have counterpart theorems in $\mathcal P$, and vice versa. Facts about partitions   map to facts about integers,  which may  be   difficult to deduce without appealing to partition theory.}
\vskip.1in
First applications of this philosophy were chiefly connected to the theory of Dirichlet series through the study of new classes of {partition zeta functions} \cite{ORS, Robert_zeta, SS-zeta}, and to $q$-series and quasimodular forms  through work on the {$q$-bracket of Bloch and Okounkov} \cite{Robert_arithmetic, Robert_jtp}. In  \cite{OSW, OSW2}, Ono, Wagner and the fifth author make a computational   application   using partition-theoretic $q$-series to give  limit formulas  for arithmetic densities   that are comparable to those using Dirichlet series; this is generalized in \cite{Robert_abelian}. A particularly stunning  application of this philosophy  is  made by Craig, Ono and van Ittersum \cite{Ono1}, in which the partition-theoretic coefficients of certain quasimodular forms are proved to be instrinsically connected to the prime numbers.\footnote{Follow-up works such as \cite{Craig, KG} apply similar ideas to   prime powers.}

In \cite{Robert_arithmetic, Robert_PhD}, a {\it partition multiplication} operation is introduced; this comes with a concept of ``subpartitions''   analogous to integer  divisors and a theory of  partition Dirichlet convolution.

\begin{definition}
For two partitions $\lambda, \gamma \in \mathcal{P}$, we define the {\it partition product} $\lambda\cdot \gamma\in\mathcal P$ (or simply $\lambda \gamma$) to be the partition obtained by concatenating the parts of $\lambda$ and $\gamma$ (and then reordering by size to align with notational convention).  For instance, $(5,3,2)\cdot (4,3,1,1)=(5,4,3,3,2,1,1).$\footnote{Equivalently, one sums  the multiplicities of corresponding parts in the two partitions.} 
\end{definition}


A new multiplicative partition statistic is introduced in \cite{Robert_zeta} that is central to the study of partition zeta functions, a multiplicative analogue of the   size $|\lambda|$ called the {\it norm} $N(\lambda)$ of $\lambda \in \mathcal P$. 

\begin{definition}
The {\it norm} of an integer partition $\lambda = (\lambda_1, \lambda_2, \lambda_3, \dots, \lambda_r)$ is the product of its parts:
\begin{equation}\label{normdef}
N(\lambda)\  :=\  \lambda_1  \lambda_2  \lambda_3 \cdots  \lambda_r\  =\  1^{m_1} 2^{m_2} 3^{m_3} \cdots  i^{m_i}\cdots \in \mathbb Z^+; 
\end{equation}
we define $N(\emptyset):=1$ (it is an empty product). \end{definition}

We note  that $N(\lambda \gamma)=N(\lambda)N(\gamma)
$ for $\lambda, \gamma\in\mathcal P$; see \cite{SS-norm} for more about the partition norm.
Along these lines, the  multiplicative and additive branches of number theory enjoy   further analogies.\footnote{For further references at the intersection of partitions and multiplicative number theory, see e.g. \cite{Alladi, Kumar, Merca}.}

\subsection{Isomorphism between partitions and natural numbers}
In \cite{supernorm}, the present authors Dawsey, Just, and Schneider define another multiplicative partition statistic, the {\it supernorm} $\widehat{{N}}(\lambda)$.

\begin{definition}
The {\it supernorm} of an integer partition $\lambda = (\lambda_1, \lambda_2, \lambda_3, \dots, \lambda_r)$ is the product 
 \begin{equation}\label{supernormdef}
 \widehat{N}(\lambda)\  :=\  p_{\lambda_1}p_{\lambda_2}\cdots p_{\lambda_r} \  =\  2^{m_1} 3^{m_2} 5^{m_3} \cdots p_i^{m_i}\cdots \in \mathbb Z^+,\end{equation}
where $p_i\in \mathbb P$ is the $i$th prime number, $i\geq 1$, and $m_j=m_j(\lambda)\geq 0$ as above; we define $\widehat{N}(\emptyset):=1$. \end{definition}

We note   that $\widehat{N}(\lambda \gamma)=\widehat{N}(\lambda)\widehat{N}(\gamma)
$ for $\lambda, \gamma\in\mathcal P$. In \cite{supernorm}, it is proved that $\widehat{N}$ induces an isomorphism of monoids. 

\begin{theorem*}[Dawsey--Just--Schneider]\label{prop0}
The supernorm map $\widehat{{N}}\colon \mathcal P \to \mathbb  Z^+$ 
is an isomorphism between   monoid $(\mathcal P, \  \cdot\  )$ with partition multiplication and  monoid $(\mathbb Z^+, \  \cdot\  )$ with integer multiplication.
\end{theorem*}

It is the central theme of \cite{supernorm}  that {\it the supernorm translates facts about partitions to analogous facts about integers, with partition parts mapping to prime factors, and vice versa.}  For example, in \cite{supernorm}, bijections between integers with certain prime factorizations, and partition formulas for natural densities of subsets of prime numbers, are proved. In a  follow-up note \cite{overpartition}, the fifth author extends partition multiplication to the set $\mathcal O$ of overpartitions (see e.g. \cite{Lovejoy})  
and proves that an extension $\widehat{N}_{\mathcal O}\colon \mathcal O \to \mathbb Q^+$ of the supernorm to overpartitions induces a {\it group} isomorphism between $(\mathcal O, \  \cdot\  )$ with partition multiplication and $(\mathbb Q^+, \  \cdot\  )$ with rational multiplication. 
In \cite{Lagarias}, J. Lagarias  proves further, striking properties  of the supernorm that uniquely characterize the map  $\widehat{N}$ in a lattice-theoretic context. Moreover,  Lagarias and Sun prove interesting asymptotic and statistical relations using sums of the reciprocal of the supernorm over different finite subsets of $\mathcal P$ in  \cite{Lagarias2}.

\subsection{Comparing the norm and supernorm}\label{compare}
Immediately one wonders about the magnitude of the supernorm compared to the partition norm; some analysis to this effect is undertaken in \cite{supernorm}. Here are some easy relations. First, a partition $\lambda$ with no part equal to 1 respects the   inequalities
\begin{equation}\label{ineq2} \ell(\lambda)\  \leq\   |\lambda|\  \leq\  N(\lambda)\  \leq \widehat{N}(\lambda).\end{equation}

A simple observation about the supernorm versus the norm  that one can make from  direct computation 
is as follows:  {\it partitions with greater lengths or having many parts equal to $1$ have supernorms that are much larger than their norms; 
 furthermore,  partitions with fewer parts and no part equal to $1$ have supernorms closer to their norms}. This empirical observation -- that seems evident but we have not proved -- is compatible with  Rosser's theorem  \cite{Rosser}, which says for $n\geq 1$,
 \begin{equation}\label{Rosser}
     p_n\  >\  n\log n;
 \end{equation}
thus, for all $\lambda \in \mathcal P$,
\begin{equation}\label{normestimate}
     \widehat{N}(\lambda)\  >\   2^{m_1(\lambda)}\prod_{i\geq 2}(i \log i)^{m_i(\lambda)}\  =\  N(\lambda) \cdot 2^{m_1(\lambda)}\prod_{i\geq 2}(\log i)^{m_i(\lambda)}. 
 \end{equation}
The right-hand side of \eqref{normestimate} is closer to $N(\lambda)$ when $\lambda$ has fewer parts and no part equal to $1$.


    Moreover, in \cite{supernorm} it is proved that if $\lambda$ has no part equal to 1, then 
\begin{equation}\label{normestimatethm} N(\lambda)\leq \widehat{N}(\lambda)\leq N(\lambda)^{{\log 3}/{\log 2}},\end{equation}
noting $\operatorname{log} 3 / \operatorname{log} 2 = 1.5849\dots$, so the supernorm of  the partition is not far off in order of magnitude from    the norm.   The inequalities in  \eqref{normestimatethm} tell us     {\it partitions  having  fixed norm and no part equal to $1$  are mapped into a relatively small interval on the number line by the supernorm}.     

We observe further that  if $\lambda$ is an unrestricted partition of size  $|\lambda|=n$ and $p_n$ is the $n$th prime number, then the supernorm falls in the interval  \begin{equation}\label{supernorminterval4.2} p_n\leq \widehat{N}(\lambda) \leq 2^n,\end{equation}
with equality on the left-hand side when $\lambda=\left<n^1\right>$, and equality on the right when $\lambda = \left<1^n\right>$; this is  proved in \cite{supernorm}. By similar reasoning, if $\lambda$ is a partition with no part equal to 1, $|\lambda|=n$, then 
\begin{equation}\label{supernorminterval4.3} p_n\leq \widehat{N}(\lambda) \leq 3^{n/2},\end{equation}
with equality on the right when $n$ is even and  $\lambda = \left<2^{n/2}\right>$. The above      inequalities  suggest    {\it partitions of  fixed {\it size}  are mapped across a relatively {\it large} interval on the number line by the supernorm},   by comparison with partitions of fixed norm. 

Combining the   inequalities above, we   see for $\lambda$ a partition with no part equal to $1$, 
\begin{equation}\label{normestimatethm4} p_{|\lambda|}\  \leq\  \widehat{N}(\lambda)\  \leq \  N(\lambda)^{{\log 3}/{\log 2}}.\end{equation}

One further inequality    one  can observe  computationally,    which is proved by J. E. Cohen \cite{Cohen} as a consequence of Rosser's theorem,\footnote{The authors are         grateful to Abhimanyu Kumar for pointing us to Cohen's proof of this inequality.} is that for $\lambda\in \mathcal P$ such that $N(\lambda)\geq 5$,
\begin{equation}\label{almostalways} 
p_{N(\lambda)}\leq \widehat{N}(\lambda).\end{equation}
 Inequality \eqref{almostalways} says that for $n\geq 5$,  {\it partitions of fixed norm  equal to $n$  have supernorms lying above $p_n$}  (similarly to   fixed size above). 

From these     inequalities,   the central observation of this paper  comes into     view: {\it comparison of the partition norm and supernorm imposes constraints on the  prime numbers.} 
 For instance, from      \eqref{normestimatethm4}   one can     deduce an upper bound for the $n$th prime number.

\begin{proposition}\label{prop}
    For $n\geq 2$, we have that 
    $$p_n\  \leq\   n^{{\log 3}/{\log 2}}.$$
\end{proposition}

\begin{proof}
 Take $\lambda = 
 (n),\  n\geq 2$; then substitute $|\lambda|=N(\lambda)=n$ on the   left and right  sides of \eqref{normestimatethm4}.  
\end{proof}

In Section \ref{sect2}, we formulate a model of   primes  based on the above observations about partitions.

\section{Partition model  of prime distribution}\label{sect2}

\subsection{Partition-theoretic model of prime gaps} 


We use our preceding  comparisons of the partition norm and supernorm to formulate a model of prime numbers. 

While the set of primes is highly enigmatic in number theory and known for its seemingly  random distribution, from the perspective of the supernorm, prime numbers have somewhat regular behavior: the $n$th prime $p_n$ represents the supernorm  $\widehat{N}(\lambda)$ of the partition $\lambda = (n)$ with a {\it single part} $n\in \mathbb Z^+$ (recalling $\emptyset$ maps to $1$). The sequence of primes is   the image of the sequence of natural numbers under the map $\widehat{N}$, while   composite integers   fill in the gaps in a more complicated way.

From this  perspective, the complexities of prime distribution are a manifestation of the complex proliferation of {\it partitions with multiple parts}, which are mapped by the supernorm to   {\it gaps between primes}. 
Thus prime gaps have a combinatorial interpretation: they are  the images of certain subsets of $\mathcal P$ under the supernorm map. This implies the following statement.

\begin{proposition}\label{measure}
Measuring prime gaps is equivalent to enumerating     partitions that map into the   respective prime gaps under the supernorm  $\widehat{N}$:
\begin{equation*}\label{gapequiv}
p_{n+1}-p_n \  \  =\  \  \    \#\left\{\lambda\in\mathcal P\  :\  p_n \leq \widehat{N}(\lambda) <p_{n+1}\right\}.\end{equation*} 
\end{proposition}

\begin{proof}
This is immediate since the supernorm is a bijection between $\mathcal P$ and $\mathbb Z^+$.  
\end{proof}

Our goal based on Proposition \ref{measure} is to {identify partitions that map into each prime gap} so we can count them. This goal suggests a potentially workable heuristic.\\

\noindent {\bf Heuristic.}  {\it Make an educated guess as to which partitions   map  into the $n$th prime gap  under the supernorm $\widehat{N}$, in order to estimate the value of $p_{n+1}-p_n$.}\\

Below, we will refine this  prime gap  heuristic 
from elementary considerations 
to formulate an initial, if overly simplistic,   partition model of prime numbers that   agrees with the facts quite well. 
We  then address sources of error to improve the model significantly. 

\subsection{Simplifying assumption}
We  wish to simplify the problem as much as possible. The    goal is to  estimate the length of the interval   $[p_n,p_{n+1})$ with high  accuracy  for each $n\geq2$.   
That is, we want to enumerate partitions   we  expect to map to an interval just above $p_n$   under the supernorm. We summarize   observations   we made in  Section \ref{compare} that may aid us in identifying such partitions:
\begin{enumerate}

\item Partitions $\lambda$ with a smaller number $\ell(\lambda)$  of parts and with no part equal to $1$ 
should have supernorms   closer to their norms in magnitude, based on   the considerations preceding \eqref{Rosser}.

\item Partitions  $\lambda$ with   {fixed norm} $N(\lambda)=n$  and with no part equal to $1$ should have supernorms of roughly  comparable magnitudes to each other, based on consideration of \eqref{normestimatethm}.

\item Partitions  $\lambda$   with fixed norm $N(\lambda)=n\geq 5$ and with no part equal to $1$  respect the inequality $p_{n}\leq \widehat{N}(\lambda)  \leq \  n^{{\log 3}/{\log 2}}$, based on consideration of  \eqref{normestimatethm4} and \eqref{almostalways}.
\end{enumerate}

Taken together, the preceding observations produce a qualitative  model   in our  minds: {\it The number line just above $p_n$ is  dominated  by the images under $\widehat{N}$ of partitions of norm equal to $n$, having no part equal to $1$, and possessing a   small number of parts.} We note   partitions with no $1$'s yield {\it odd} supernorm values; for each odd value, we must also count the even number that   follows.  Furthermore, the norm-$n$ partitions with {\it many} parts are mapped by $\widehat{N}$ closer to the upper bound $n^{{\log 3}/{\log 2}}$ and do not map into the gap $[p_n , p_{n+1})$, but rather give rise to a lengthy  ``tail'' of leftover integers that contribute to subsequent prime gaps more diffusely, perhaps        somewhat randomly. The interval $[p_n , p_{n+1})$ also includes integers   arising from the   ``tails'' of subsets of partitions with norms less than $n$;  we anticipate their contribution  is negligible asymptotically.

To translate this qualitative vision into a computational model one can test and use to make predictions, we  shall  assume  the following    extreme  simplification to address the observations above.\footnote{In \cite[Appendix]{supernorm1}, the second, fourth, and fifth authors give more detailed discussion of this  assumption.} 

\begin{assumption}  Assume for simplicity that all odd integers in the interval $[p_n, p_{n+1}),$ $n\geq 2$,  are the images of partitions with norm equal to $n$, with  no part equal to $1$,   having    one or two parts, under the supernorm $\widehat{N}$. Then  
\begin{equation}\label{gapequiv2}
p_{n+1}-p_n \  \  =\  \  2\cdot  \#\left\{\lambda\in \mathcal P\  :\   \  N(\lambda)=n, \  m_1(\lambda)=0,\    \ell(\lambda)=1\  \text{or}\  2\right\}.\end{equation}
\end{assumption} 

\begin{remark}
    That is, we will assume the images of the above partitions dominate the number line immediately above $p_n$.
\end{remark}

Integers having   two prime factors are called {\it semiprimes}. We expect  this    simplification yields   an {\it underestimate} of  prime gaps  from the incorrect     assumption   all   odd numbers are   prime or semiprime. 

Now observe that  the  partitions of norm  $n\geq 2$, with no part equal to 1 and having one or two parts, are precisely the set consisting of the partition $(n)$ into one part,  together with  partitions $(d_1,   d_2)$ where $d_1 d_2=n, \  d_1 \geq d_2$; that is,   partitions whose parts are pairs of divisors of $n$. There are $\left\lceil d(n)/2\right\rceil$ such divisor pairs where $d(k)$ is the {\it divisor function}, setting $d(0):=0$,\footnote{We define $d(0):=0$ for our purposes here, as there are no positive integers $d_1, d_2,$ such that $d_1 d_2=0$.} including $d_1=n, d_2=1,$ which we associate to  the partition $(n)$.     Then Simplifying Assumption 1 translates to the   relation 
\begin{equation}\label{estimate1}
    p_{n+1}-p_n\  =\  2\left\lceil \frac{d(n)}{2}\right\rceil,
\end{equation}
where $\left\lceil x\right\rceil$ is the ceiling function. 

We refer the reader to Table \ref{table2} to see   \eqref{estimate1}  does correctly predict most of the prime gaps between 1 and 100. 
Computations show this near-correctness  diminishes as $n$ increases, but as we   will prove, 
equation \eqref{estimate1} is compatible with the prime number theorem asymptotically. 

\subsection{Initial model} The following formulation is suggested by    discrepancies between the   partition norm and  supernorm noted in Section \ref{compare}, which led   to  Simplifying Assumption 1 above. Since $p_n=1+\sum_{k=0}^{n-1}(p_{k+1}-p_k)$ by telescoping series, defining $p_0:=1$ as we did above, then replacing each summand $p_{k+1}-p_k$ by the estimate $2\left\lceil {d(k)}/{2}\right\rceil$ from   \eqref{estimate1} leads us to a model of prime numbers.  \\

\noindent {\bf Model 1.} {\it The prime numbers $p_1, p_2, p_3,$ $\dots, $ can be modeled by the sequence having initial value $p_1=2$ and for $n\geq 2$ having the values 
$$p_{n}\  =\  1\  +\  2\sum_{k=1}^{n-1}\left\lceil \frac{d(k)}{2}\right\rceil.$$}

\begin{remark}
    For   computational ease, note    $2\left\lceil d(k)/{2}\right\rceil=d(k)$ for all $k\neq m^2, m\in \mathbb Z^+$. On the other hand, if $k$ is a perfect square then $2\left\lceil {d(k)}/{2}\right\rceil=d(k)+1$. Thus for $N\geq 1$,  we have  that 
\begin{equation}\label{simplify}
2\sum_{k=1}^{N}\left\lceil \frac{d(k)}{2}\right\rceil\  =\   \sum_{k=1}^{N} d(k)\  +\   \left\lfloor \sqrt{N}\right\rfloor\  =\  \sum_{k=1}^{N}\left\lfloor \frac{N}{k}\right\rfloor\  +\    \left\lfloor \sqrt{N}\right\rfloor,
\end{equation}
   by well-known methods related to the summatory function of $d(k)$ \cite{Apostol1}. This yields an equivalent formula in Model 1  that does not require explicitly computing the divisor function: 
   \begin{equation}
p_n\  =\  1+\sum_{k=1}^{n-1}\left\lfloor \frac{n-1}{k}\right\rfloor\  +\  \left\lfloor \sqrt{n-1}\right\rfloor.
\end{equation}
\end{remark}

The first ten terms of this sequence of estimated values for $p_n$ are $$2, 3, 5, 7, 11, 13, 17, 19, 23, 27,$$   
and from here Model 1 does not give the correct sequence of   prime values. However, the model does {\it almost} give the correct sequence of {\it prime gaps} at small numbers; see Table \ref{table2}. 

\begin{table}[h]    \centering
    \begin{tabular}{| c | c | c | c |   }
\hline
 $n$\  & $p_n$ & {$p_{n+1}-p_n$ (Actual)} & {$p_{n+1}-p_n$ (Model 1) }\\ \hline
 {1}&  {2} & {1}& {1} \\ \hline
 {2}&{3} & {2}&{2} \\ \hline
 {3}&{5} &{2}&{2} \\ \hline
 {4}& {7}&{4} &{4} \\ \hline
 {5}& {11}&{2}& {2}\\ \hline
 {6}&{13}&{4}& {4}\\ \hline
 {7}&{17}&{2}& {2} \\ \hline
 {8}&{19}&{4}  & {4} \\ \hline
 {9}& {23} &{6}& \colorbox{lightgray}{4} \\ \hline
 {10}& {29} & {2}& \colorbox{lightgray}{4} \\ \hline
 {11}& {31} & {6}& \colorbox{lightgray}{2} \\ \hline
 {12}& {37} & {4}& \colorbox{lightgray}{6} \\ \hline
  {13}& {41}&{2}& {2} \\ \hline
  {14}&  {43}&{4} & {4} \\ \hline
 {15}& {47}&{6}&  \colorbox{lightgray}{4}\\ \hline
  {16}& {53}& {6} &{6}\\ \hline
  {17}&  {59}&{2} & {2} \\ \hline
  {18}& {61}&{6}& {6} \\\hline
 {19}& {67}&{4}&  \colorbox{lightgray}{2} \\ \hline
 {20}& {71}& {2} &  \colorbox{lightgray}{6} \\ \hline 
 {21}& {73}& {6} & \colorbox{lightgray}{4}\\ \hline
  {22}& {79}& {4} & {4} \\ \hline
 {23}& {83}& {6}&  \colorbox{lightgray}{2} \\\hline
  {24}&  {89}&{8}& {8} \\ \hline
  {25}& {97}&{4} & {4} \\ 

\hline
\end{tabular}\vskip.1in
    \caption{Comparing actual prime gaps to predictions from Model 1; we highlight entries where the prediction for the $n$th prime gap is off}
    \label{table2}
    \end{table}

One can   estimate $\pi(n)$   by counting the primes up to $n$ predicted by the model. Table \ref{table4} and Table  \ref{table2} display that  Model 1 does capture   aspects of prime distribution qualitatively: at small numbers, it emulates the unpredictable prime gaps reasonably well, and at large numbers, it yields  numerical estimates for $\pi(n)$ that are   comparable to   the prime number theorem. Note from Table \ref{table4} that Model 1 evidently leads to an {\it overestimate} for $\pi(n)$.\footnote{We interpret this from our premises as   a result of the model's representing an underestimate of prime gaps.}

In the next subsection, we  discuss explicit   predictions of Model 1 and try to interpret them. 

\begin{remark}
We note  that J. N. Gandhi proves a formula for the $n$th prime number in \cite{Gandhi}, which  represents a base-$2$ version of the sieve of Eratosthenes \cite{Gandhi2}.\footnote{The authors are grateful to Ken Ono for pointing us to Gandhi's formula.} 
\end{remark}

\subsection{Predictions from Model 1}\label{prob}

Below we   show that Model 1 implies 
the main asymptotic term of the prime number theorem and produces   unusual predictions about local fluctuations in prime gaps, including the infinitude of twin primes as a special case. 




The {\it prime number theorem} was   conjectured by both Gauss and Legendre in the 1790s based on newly published tables of primes, and was known as the ``prime number conjecture'' for a century, until  Hadamard and de la Vall\'{e}e Poussin proved it independently at the very end of the nineteenth century \cite{Tenenbaum}. The prime number theorem is usually expressed as an asymptotic estimate for $\pi(n)$,  
\begin{equation}\label{pnt1}
    \pi(n)\  \sim\  \frac{n}{\log n}
\end{equation}
for $n\geq 2$ as $n$ increases, where $\log x$ denotes the natural logarithm;  or by the asymptotically equivalent (and evidently more accurate) estimate provided by the {\it logarithmic integral} $\operatorname{li}(n)$: 
\begin{equation}
    \pi(n)\  \sim\  \operatorname{li}(n):= \int_{2}^{n}\frac{dt}{\log t}.
\end{equation}
The prime number theorem is equivalently  formulated as an asymptotic estimate for the $n$th prime: 
\begin{equation}\label{pnt}
    p_n\  \sim\  n \log n. 
\end{equation}
The estimate \eqref{pnt} follows from \eqref{pnt1} by observing that  $\pi(p_n)=n\  \sim\   p_n/\log p_n\  \sim\  p_n/\log n$.  
This final  formulation of the prime number theorem is suggested naturally by  Model 1. The first prediction of the model is a statement of fact. 

\begin{prediction}[Theorem]\label{cor2}
Model 1  predicts the 
prime number theorem's estimate for the $n$th prime,  $$p_n\  \sim\  n\log n\  \  \text{as}\  \  n\to \infty.$$ \end{prediction}

\begin{proof}
    It follows from Simplifying Assumption 1 that  \begin{equation}\label{newgap00}
p_n\  =\  1+\sum_{0\leq i \leq n-1}(p_{i+1}-p_{i})\  \sim \sum_{1\leq i \leq n-1}d(i)\  =\  \  (n-1) \log (n-1)+(n-1)(2\gamma-1) +O(\sqrt{n})\end{equation} where $\gamma=0.5772\dots$ is the Euler-Mascheroni constant; the final equality is due to Dirichlet \cite{Apostol1}. The far right-hand side of \eqref{newgap00} is asymptotic to $n \log n$ as $n$ increases.
\end{proof}
\begin{remark}
    We note    there is a second-order error term of  magnitude $n\log \log n$ implied by the prime number theorem, arising from $n=\pi(p_n)\sim p_n/\log p_n$, since $\log p_n \sim \log(n \log p_n)  \sim \log n + \log\log n$, that is missing from  Prediction \ref{cor2}. We address this missing error  in Section \ref{sect5}. 
\end{remark}

An unusual feature of Model 1 is that it suggests {\it predictable local behaviors of prime gaps}.

\begin{prediction}\label{cor3}
For $n\geq 2$, Model 1 predicts the $n$th prime gap   is larger or smaller, depending on if $n$ has a larger or smaller number of divisors, respectively. \end{prediction}

\begin{proof}
    This is   immediate     from \eqref{estimate1}, as $d(n)$ effectively controls the   prime gap in this model. 
\end{proof}

Prediction \ref{cor3} is   counterintuitive:   it says that {\it the indices of primes influence the prime gaps}. Table \ref{table2} gives evidence of this up to $n=100$, where the majority of   prime gaps match the predicted values exactly, but further computations indicate this exact matching does not hold as $n$ increases. 
On the other hand, it follows from the prime number theorem that the average number of prime gaps continues to match Model 1's predictions asymptotically. Interpreted strictly, Model 1 says the $n$th prime gap should be close to $d(n)$; Prediction \ref{cor3} is  a more probabilistic interpretation. 

The twin prime conjecture is a natural consequence of Model 1.

\begin{prediction}\label{cor4}
Model 1 predicts twin prime pairs occur at prime-indexed primes, i.e.,  at primes $p_n$ such that $n\in \mathbb P$. Moreover, Model 1 predicts the set of twin primes is infinite.\end{prediction}

\begin{proof}
    This prediction focuses on a special case of Model 1, the case  in which  $d(n)=2$ if and only if  $n\in \mathbb P$. Since prime-indexed primes are associated to twin prime pairs in Model 1,   the predicted infinitude of twin primes follows from the infinitude of the primes proved by Euclid. 
\end{proof}

Known observations on finite intervals are suggestive of the infinitude prediction, although as of this writing, the twin prime conjecture is undecided. However,  since we expect that Model 1 is an underestimate of most  prime gaps, it cannot claim to model twin primes with high accuracy;  the model does not claim the precision to definitively predict a gap of minimal size $2$, ruling out  possible contributions from partitions with more than two parts mapping into the gap under $\widehat{N}$. We wonder if  Prediction \ref{cor4} can be interpreted probabilistically like Prediction \ref{cor3}:  {\it prime-indexed primes are more likely to belong to twin prime pairs by comparison with arbitrary prime numbers}. 

In Section \ref{sect4}, we make a preliminary check of these predictions. 

\begin{remark}
    One can formulate  further, similar predictions  from Model 1. 
    For example, consider a prime number $p_m$ whose index $m \in \mathbb P$ is itself the first member of a pair of twin primes, i.e., $m$ and $m+2$ are both prime. The model predicts $p_m$ will be the first  member  of a prime $4$-tuplet, meaning that $p_m,\  p_{m+1}$ are twin primes and $p_{m+2},\   p_{m+3}$ are also twin primes (see e.g.  \cite{Forbes}). 
    \end{remark}


\subsection{Critique of Model 1}
As   noted previously, while Model 1   emulates certain   features of prime distribution, we expect  {\it a priori} that it represents an underestimate of  prime gaps in general and an overestimate for $\pi(n)$,  since the model does not account for odd numbers with more than two prime factors; this expectation is supported by our data.   
Moreover,   the model is missing a second-order  error term of size $n\log\log n$  by comparison with  the prime number theorem's  estimate for $p_n$. 

While the model impressed itself, so to speak, on the authors from   observations about the partition norm and supernorm, the considerations   preceding \eqref{Rosser} involve an empirical observation that is not proved rigorously, although we noted it is  compatible with Rosser's theorem. 
Furthermore, we   have not   proved (or  observed) that  {\it even one} partition with norm $n$ should necessarily map into the $n$th prime gap under the supernorm; this is an assumption justified by its predictive success. 
Model 1 also does not suggest     that  composite   integers should respect anything like   the correct integer ordering; nor does it  account for constraints on primes imposed by multiplication. 

We   test  Predictions \ref{cor3} and \ref{cor4} in Section \ref{sect4}. In Section \ref{sect5}, we examine the error in Model 1 and make further simplifying assumptions that lead to a  computational  model with greater   accuracy.

\section{Testing the model}
\label{sect4}

\subsection{Testing predictions from Model 1}  Tables \ref{table4} and \ref{table2} and Prediction \ref{cor2} provide reasonable support for Model 1. Predictions \ref{cor3}
 and \ref{cor4} place unusual emphasis on the factorizations of the {\it indices} of prime numbers. 
 We give these unusual predictions a  preliminary check.  

\subsection{Methods}
 We utilize the   high-performance computing shared facility at Michigan Technological University to check Predictions \ref{cor3} and \ref{cor4}. For    calculations we use  Wolfram Mathematica computer algebra system and Python programming language. The authors made use of OpenAI GPT-4o large language model software    to assist with writing Python code to generate experimental data exclusively in Sections \ref{test1} and \ref{test2} below; we checked and revised this code as needed and take full responsibility for the results. We will share our computer code and data upon request.
 \subsection{Computational test of Prediction \ref{cor3}}\label{test1} 
 Prediction \ref{cor3}  says the $n$th prime gap is larger or smaller depending on the size of $d(n)$. As we noted above, Table \ref{table2} is consistent with this prediction because $p_{n+1}-p_n$ is equal to $2\lceil d(n)/2 \rceil$ {\it exactly} for most of the prime gaps  in the table. However, further computations reveal  the exact equalities decrease in frequency as $n$ increases.  
 
 We   use a  less exact method to check this  prediction. It follows from the prime number theorem that the {\it average order} of the $n$th prime gap is asymptotic to $\log p_n\sim \log n + \log \log n$ as $n\to \infty$. As is standard (see e.g. \cite{merit}), we define  the  {\it merit} $M(n)$ of the $n$th prime gap to be  the   ratio   \begin{equation}\label{M}  M(n)\  :=\  \frac{p_{n+1}-p_n}{\log p_n}.
 \end{equation}  
 When the merit is greater or less than one, the prime gap is larger or smaller than average, respectively. Now, the average order of $d(n)$ is asymptotic to $\log n$, and likewise for $2\lceil d(n)/2\rceil$ that is equal to  $d(n)$ except at perfect squares when it equals $d(n)+1$. For $n\geq 2$, let us    define the {\it Model 1 merit}   $M_1(n)$ to be  the analogous ratio with respect to the $n$th {\it modeled} prime gap,
 \begin{equation}  M_1(n)\  :=\  \frac{d(n)}{\log n};
 \end{equation} 
 we omit the ceiling function for computational ease  since $d(n)=2\lceil d(n)/2\rceil$ almost always.  When the Model 1 merit is greater or less than one, the divisor function is larger or smaller than average.   

To give a preliminary check of  Prediction \ref{cor3}    --  whether prime gaps   tend to increase or decrease with $d(n)$ -- we count how often $M(n)$ and $M_1(n)$
are  simultaneously greater or less than one.\  
 
\begin{result}
{We compute $M(n)$ and $M_1(n)$ for $n\leq 1{,}000{,}000$. We find     $M(n)>1$   in $36.01\%$ of instances and $M(n)<1$   in $63.99\%$ of instances. We find   $M_1(n)>1$ in $37.94\%$ of instances and $M_1(n)<1$ in $62.06\%$ of instances. Finally, we find    $M(n)$ and $M_1(n)$ are simultaneously greater or less than one in $53.34\%$ of instances.}
\end{result}

\begin{remark}
   The   similarity between corresponding   $M(n)$ and $M_1(n)$ percentages is noteworthy.
   For instance, one   expects $M(n)<1$ to occur often since smaller prime gaps are more abundant than larger gaps up to $n$ (see e.g. \cite{Tao2}).   One also expects     $M_1(n)<1$  often since the {\it normal} order of $d(n)$ is well known to be  asymptotic to $\log \log n = o(\log n)$ as $n\to \infty$. But we are not aware of a   well-known      reason     $p_{n+1}-p_n$ should be smaller than average  with  similar  frequency to the divisor function. Perhaps an analysis along the lines of that given in \cite{Tao2} would provide an explanation.
\end{remark}

Up to  one million, a slight majority of       prime gaps    increase or decrease   with the divisor function.

\subsection{Computational test of weak version of Prediction \ref{cor4}}\label{test2}  In   Section \ref{prob}  we make a  probabilistic conjecture based on Prediction \ref{cor4}, that a prime-indexed prime is more likely to begin a twin prime pair than is an arbitrary prime number.     
Table \ref{table4} supports this weak version of Prediction \ref{cor4} up to $n=100$; there are 6 out of 8 twin prime pairs correctly predicted in the table (but these exact results do not continue as $n$ increases)  and an additional three are predicted erroneously. 

Furthermore, one can deduce from Cram\'{e}r's random model of primes that the number of twin prime pairs less than or equal to $n$ is asymptotic to $n/(\log n)^2$   \cite{Tenenbaum}. The number of prime-indexed primes less than or equal to $n$ is    $\pi(\pi(n))\sim\pi(n)/\log \pi(n) \sim n/(\log  n)^2$ by the prime number theorem so, in fact,   {\it we do expect twin prime pairs and prime-indexed primes   to be  approximately equinumerous at large numbers}. How often do twin prime pairs and prime-indexed primes co-occur? 

Let us refer to a prime gap whose first member has prime index as a {\it prime-indexed prime gap}, and    to a twin prime pair whose first member has prime index as a {\it prime-indexed twin prime pair}.

We run a few cursory experiments to explore statistics related to Prediction \ref{cor4}. First, we give a simple ``greater-or-less'' test like the one in    Section \ref{test1} related to the merit $M(n)$. Model 1 predicts $M(n)=2/\log p_n<1$ every time $n$ is prime  since the $n$th prime gap is  $2\lceil d(n)/2\rceil$ in the model.

\begin{result}
We only compute   merits   of prime-indexed prime gaps up to one million, i.e.,   we compute $M(p)$ over primes  $p\leq \pi(1{,}000{,}000)=78{,}498$. We find     $M(p)>1$   in $37.39\%$ of instances and $M(p)<1$   in $62.61\%$ of instances. Finally, up to one million we find $10.59\%$ of prime-indexed prime gaps are also  twin prime pairs, and $9.98\%$ of twin prime pairs begin with prime-indexed primes. Data suggests these final  percentages    decrease as $n\to \infty$; see  Table \ref{table3}.  
\end{result}

\begin{table}[h]    \centering
    \begin{tabular}{| c | c | c | c |   }
\hline
 $n$\  & {\# Twin prime pairs $\leq n$} & {\# Prime-indexed twin prime pairs $\leq n$}  \\ \hline
 {10}&  {2} & {2}  \\ \hline
 {100}&{8} & {6}  \\ \hline
 {1000}&{35} &{12}  \\ \hline
 {10,000}& {205}&{30}  \\ \hline
 {100,000}& {1224}&{154} \\ \hline
 {1,000,000}&{8169}&{816}\\ \hline 
\end{tabular}\vskip.1in
    \caption{Total number of twin prime pairs vs. those beginning with prime-indexed primes}
    \label{table3}
    \end{table}

This result is inconclusive. We   find $M(n)<1$ in the majority of cases, a result  compatible with Prediction \ref{cor4}. However,   our computational results for $M(n)$ restricted to prime values of $n$ split into  almost the same ``greater-than'' and ``less-than'' proportions as   the computations over all $n$   in Section \ref{test1}; apparently restricting   to prime indices $n$ does not increase   the $M(n)<1$ proportion.

\subsection{Second test of weak version of Prediction \ref{cor4}} Next,  we directly compare the probability that a prime-indexed prime  begins a  twin prime pair with the probability that an arbitrary prime number does so. 
Define
\begin{equation}
\operatorname{Prob_1}(n):=\frac{\#\{\text{prime-indexed primes}\  p\leq n\  \text{such that}\  p+2\   \text{is also prime}\}}{\#\{\text{prime-indexed primes}\  p\leq n\}},
\end{equation}
the probability that a random prime-indexed prime $\leq n$ begins a twin prime pair, and define
\begin{equation}
\operatorname{Prob_2}(n):=\frac{\#\{\text{primes $p\leq n$ such that $p+2$ is also prime}\}}{\pi(n)},
\end{equation}
the probability that a random unrestricted prime $\leq n$   begins a twin prime pair.   
We   compute  
\begin{equation}\label{P}
P(n):=\frac{\operatorname{Prob_1}(n)}{\operatorname{Prob_2}(n)}
\end{equation}
 as $n$ increases.  By   probabilistic reasoning, we     expect the ratio $P(n)$ to be around one:   prime-indexed primes and arbitrary primes 
 should have the same   probability of being twin primes. 
 
\begin{result}
At small numbers,     $P(n)$ is almost always greater than one with minor oscillations in the value locally, except for infrequent larger oscillations dipping   below one. 
From  about $n=10{,}000{,}000$ up to at least $10$ billion, $P(n)$ is strictly greater than one. 
See Figure \ref{fig3}. 
\end{result}

This result is compatible with Prediction \ref{cor4} but we can see  another possible  explanation for  it: {\it on average, primes become sparser and twin prime pairs become   rarer  as $n$ grows}. The 
typical prime-indexed prime less than or equal to $n$   
is on average smaller than the 
typical arbitrary prime on the same interval, and thus may be more likely to be the first of a twin 
prime pair not because it is prime-indexed, but simply because it is smaller.

In an attempt to remove this bias related to small numbers, we  look at prime-indexed primes versus arbitrary primes on the interval $[n, 2n)$.  
Let $Q(n)$ denote the analogous ratio to $P(n)$  in \eqref{P}   with both   probabilities evaluated on the interval $[n, 2n)$  instead of   $[1, n]$.\footnote{This $Q(n)$ test was suggested to the authors by Michael Filaseta. We are   grateful to Filaseta for sharing  invaluable advice on computing  probabilities such as these (Private communication, April 12--13, 2025).}  

\begin{result}
    We compute $Q(n)$ up to $n=500{,}000$. We find      $Q(n)$ oscillates irregularly around one. We note a tendency for $Q(n)$ to be greater than one up to   $n=80{,}000$ but not on larger intervals, where $Q(n)$ appears to approach one asymptotically as $n$ increases.  
\end{result}

We conclude from the computational experiments in Section \ref{test2} and in this subsection that Prediction \ref{cor4} may hold at small numbers,\footnote{However, we are not sure if this results from the prime indices or from probabilistic anomalies at small numbers.} but does not appear to hold as $n$ increases.


  \begin{figure}
\includegraphics[scale=.165]{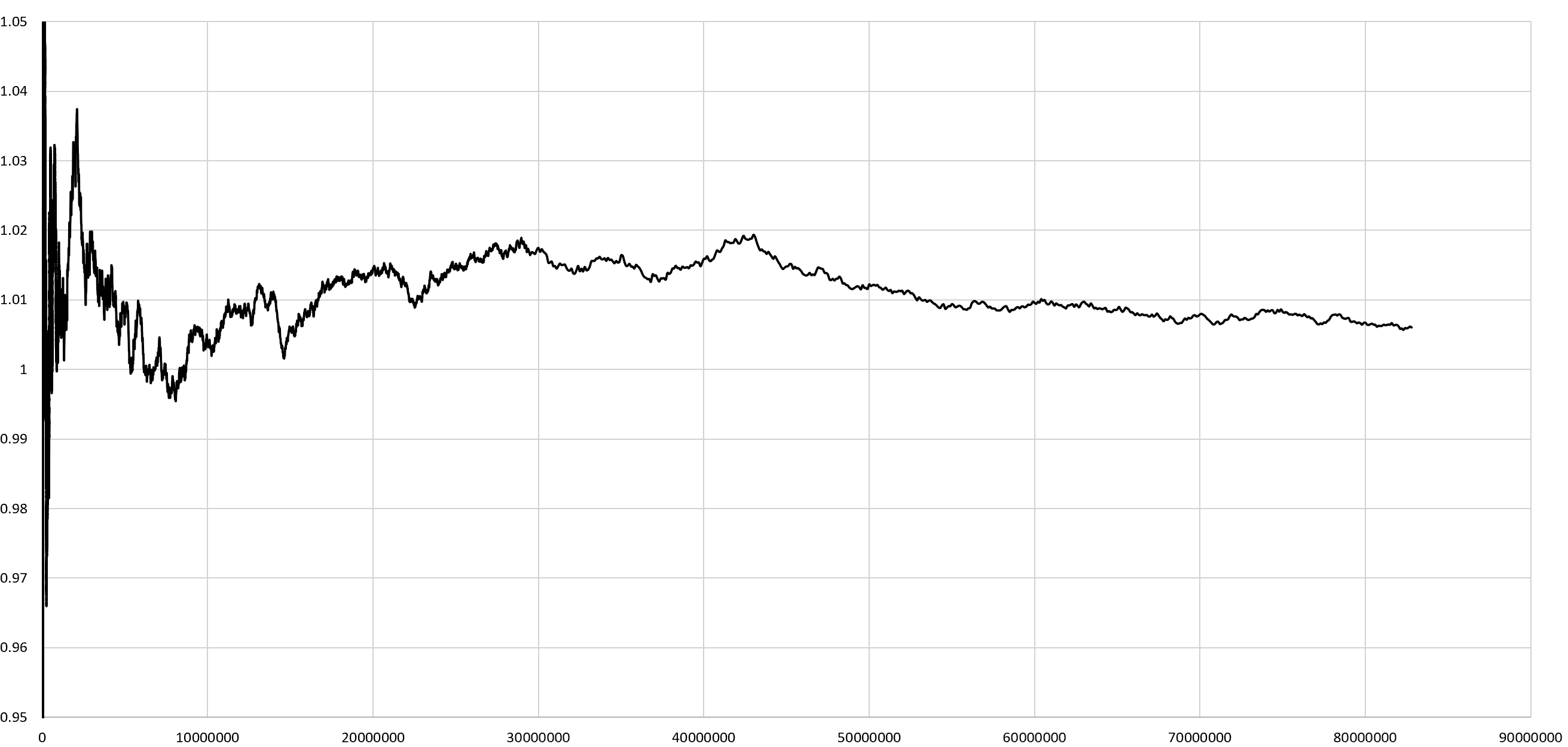}
 \caption{Plot of $P(n)$ up to $n=100{,}000{,}000$}
 \label{fig3}
  \end{figure}


\section{Improvements to the model}
\label{sect5}

\subsection{Refining the error term}

We now turn our attention from formulating a conceptual model of prime distribution to improving our computational model and its numerical estimates  for $\pi(n)$. Recall that we  define a {semiprime} to be an integer having exactly two prime factors.  Let $\pi_k(n)$ denote the number of positive integers less than or equal to $n\geq 1$   having exactly $k\geq 0$ prime factors (so   $\pi_1(n)=\pi(n)$ and  $\pi_0(n)=1$ for all $n\geq 1$).   All the integers up to $n$ are enumerated by
\begin{equation}\label{first} 
    n\  =\  1+\pi(n)+\pi_2(n)+\pi_3(n)+\dots+\pi_k(n)+\dots
\end{equation}
with only  finitely many nonzero terms, depending on $n$. Substituting $n\mapsto p_n$ in \eqref{first} gives 
\begin{equation}\label{newgap0}
p_n\  =\  1+n+\pi_2(p_n)+\pi_3(p_n)+\dots.
\end{equation}
This provides an explicit formula for the $n$th prime gap in terms of   values $\pi_k(p), p\in \mathbb P$: 
\begin{equation}\label{newgap} p_{n+1}-p_n\  =\  1+\left[\pi_2(p_{n+1})-\pi_2(p_n)\right]+\left[\pi_3(p_{n+1})-\pi_3(p_n)\right]+\dots.
\end{equation}

Now, \eqref{newgap0} and \eqref{newgap}  are not     conceptually compatible with Model 1;  the equations track factorizations of even as well as odd integers in the prime gap, while      Simplifying Assumption 1 only allows for enumerating odd semiprimes in the image of the supernorm and doubling the number. But these identities can inform   corrections to Model 1. It is an estimate due to Landau \cite{handbuch} that for $n> 1, k\geq 1$, as $n\to\infty$ we have 
\begin{equation}\label{landau}
\pi_k(n)\  
\sim\   \frac{n \left(\log \log n\right)^{k-1}}{\log n \   (k-1)!}.
\end{equation}
The $k=1$ case of \eqref{landau} is the prime number theorem. It follows for $k\geq 1$  
that\footnote{Noting   $\log p_n  \sim \log(n\log n) = \log n +\log\log n \sim \log n \sim \log(n+1)  \sim \log p_{n+1}$ as $n\to\infty$.} 
\begin{equation}\label{landau2}
\pi_k(p_{n+1})-\pi_k(p_n)
\  \sim\     \frac{\left(\log \log n\right)^{k-1}}{(k-1)!}.
\end{equation}
For fixed $n\geq 3$, this final ratio approaches zero as $k\to \infty$. Therefore, we expect integers with $k=3$ prime factors are the next most populous set in a prime gap, apart from  semiprimes.  {To improve our model, we should    model the contributions of integers with three   prime factors.} 

However, for conceptual and computational simplicity in our model, we want to avoid   considering   factorizations with  many   prime components. 
We posit a second  simplifying assumption that avoids having to perform factorizations involving more than two prime factors, based on  the following elementary observation: {\it every   integer less than $p_{n+1}$ having   three prime factors is the image of a semiprime less than $p_n$  multiplied by a prime.} We will assume for simplicity that {\it the converse also holds}, as an approximate approach to modeling the contribution of integers with three factors.

Define a {\it $k$-almost prime} to be an integer having exactly $k\geq 2$ prime factors (see e.g. \cite{Heath-Brown}). 

\begin{assumption}  
Assume for simplicity that every semiprime less than $p_n$ maps into the interval  $[2,\  p_{n+1}),\  n\geq 2,$ under multiplication by some prime number, 
and that the $3$-almost primes resulting from these products are not already enumerated by Model 1. By \eqref{landau}, as $n$ increases this yields an estimated contribution of \begin{equation}\label{asymptotic}
\pi_2(p_n)\  \sim\  \pi(p_n) \log \log p_n\  
\sim\   n\log\log n \end{equation}
integers having     three prime factors in the interval $[2, p_{n+1})$, in addition to the  odd semiprimes and their even   ``doubles''  enumerated by Model 1. 
\end{assumption}

 We   count  these integers with three  prime factors in our updated model.  {\it Simplifying Assumption 2   fills in the missing error of order $n\log\log n$   noted previously}.      Simplifying Assumption 2   likely produces an overestimate, but we have not proved this;  only semiprimes less than $p_{n+1}/2<p_n$ can   map into the interval $[2, p_{n+1})$ via  multiplication but some semiprimes might produce more than one $3$-almost prime that is less than $p_{n+1}$, through multiplication by  different   primes.  

  We    make one further simplifying assumption that is solely computational, and as such, represents an {\it ad hoc} revision -- yet it leads   to a {significantly}  more accurate   computational estimate for $\pi(n)$ that is {\it almost exact} at small numbers. One may want      to replace the sum in \eqref{newgap00} with an integral. 
 Note the   asymptotic estimate on the right side of  \eqref{newgap00} can be  expressed as      
 \begin{equation}
     p_n\  \sim\     \int_{0}^{n-1}\log t\  dt \  +\   2\gamma (n-1) .
 \end{equation}
 {This integral estimate is  simpler, and  unaffected asymptotically,   if we delete the $2\gamma (n-1)$ term.\footnote{We do not at present know why this deletion  produces better accuracy or   if it is optimal to this effect.}} 

\begin{assumption}  
Assume for simplicity   that subtracting $2\gamma (n-1)$ from the model will not hurt the reasonable accuracy of the model  at small numbers and will have negligible impact on the estimate for $p_n$ asymptotically.
\end{assumption}



Combining Simplifying Assumptions 2 and 3 above with the statement of Model 1, we posit a refined model. We   add a correction term $\left\lfloor \pi_2(p_{n-1}) - 2\gamma (n-1)  \right\rfloor$ to the formula for $p_n$, where   $\left\lfloor x \right\rfloor$ is the floor function; recall that we set  $\left\lfloor x \right\rfloor:=0$ if $x<0$ throughout this paper.   Simplifying Assumption 2 provides   number-theoretic justification for adding   the   $\pi_2(p_{n-1})\sim (n-1)\log\log (n-1)\sim n\log\log n$ term; subtracting the $2\gamma (n-1)$ term is an {\it ad hoc} correction for computational simplicity. Use of the floor function in the correction term   ensures   our outputs are   integer values. \\  


\noindent {\bf Model 2.} {\it The prime numbers $p_1, p_2, p_3,$ $\dots, $ can be modeled by the sequence having initial value $p_1=2$ and for $n\geq 2$ having the values  
$$p_{n}\  =\  1 \  +\   2\sum_{k=1}^{n-1}\left\lceil \frac{d(k)}{2}\right\rceil\  +\  \left\lfloor \pi_2(p_{n-1}) - 2\gamma (n-1) \right\rfloor.$$}

 Model 2 outputs the same  initial values  as Model 1 for the estimated sequence of primes, viz. $2, 3, 5, 7, 11, 13, 17, 19, 23, 27$,  but provides a better model of prime distribution as $n$ increases. The reader can confirm from Table \ref{table4}     that  Model 2 gives a {significantly} better numerical approximation to $\pi(n)$ than   Model 1 does; in fact, estimates from Model 2 are close to exact at smaller numbers. However, testing Model 2 numerically  as $n$ increases becomes increasingly cumbersome: one has to check factorizations for every integer up to $p_{n-1}$, which is computationally expensive and   goes against our goal of having an explicit formulation for the model.\footnote{We note for clarity that in our computations  we use actual primes $p_{i-1}$ inside the Model 2 correction term, not the estimated sequence produced by the model; the Model 2* correction term approximates both cases asymptotically.}

At large numbers,  Landau's  formula \eqref{landau}  approximates $\pi_2(n)$. We will use the relation  $\pi_2(p_{i})\sim i\log\log i,\  i\geq 2$, by  \eqref{asymptotic} to produce a version of Model 2 that is more computationally efficient.  Introducing this  asymptotic term  potentially introduces a  source of    error, but  in practice turns out to   yield  surprisingly good estimates for $\pi(n)$. \\

\noindent {\bf Model 2* (computational version).} {\it The prime numbers $p_1, p_2, p_3,$ $\dots, $ can be modeled by the sequence having initial values $p_1=2,\  p_2=3$,   and for $n\geq 3$  having the values  
$$p_{n}\  =\  1 \  +\   2\sum_{k=1}^{n-1}\left\lceil \frac{d(k)}{2}\right\rceil\  +\  \left\lfloor (n-1)\cdot\left(\log \log(n-1)-2\gamma \right)\right\rfloor.$$}

\begin{remark}
   Models 2 and 2* can be simplified computationally using equation \eqref{simplify}.
\end{remark}

One can    count the number of primes that Model 2* predicts up to any positive integer $n$. Model 2* still outputs the initial values $2, 3, 5, 7, 11, 13, 17, 19, 23, 27$, and produces an improvement on Model 2 in terms of its approximation to $\pi(n)$ as $n$ increases. 

{We refer the reader to Table \ref{table4} to see   that    Model 2* yields quite accurate estimates for $\pi(n)$.}

\subsection{Comparison of relative errors in the models}  From Figures \ref{fig4}, \ref{fig5}, and \ref{fig6}, one can compare the  relative errors in the estimates for $p_n$ provided by Models 1, 2, and 2*, respectively.\footnote{Relative error is computed by   ${|v_{\operatorname{est}}-v_{\operatorname{act}}|}/{v_{\operatorname{act}}}$ with $v_{\operatorname{est}}\geq 0$    the estimated value and $v_{\operatorname{act}}>0$   the actual value.} As the graphs indicate, at $n=100{,}000$, the  relative error in Model 1 is about $0.1$, the  relative  error in Model 2 is about $0.015$, and the relative  error in Model 2*  is below $0.005$.

\subsection{Predictions from Model 2}\label{model2pred}

With   additional correction terms whose necessity largely follows from the premises of Model 1 (but   that also include  an {\it ad hoc} adjustment), Models 2 and 2* give surprisingly good estimates for $\pi(n)$ at small numbers.  The second-order corrections do not impact the   models' compatibility with the prime number theorem  asymptotically. As Table \ref{table4} displays, Models 2 and  2* are both improvements on Model 1 from a computational perspective.

However, Model 2 requires one to enumerate semiprimes, whereas Model 2* together with the right-hand side of  \eqref{simplify} provides a fully  analytic   model of primes requiring no number-theoretic calculations beyond division. Since it involves simpler computations and is  evidently    more accurate in its estimates of $\pi(n)$ as $n$ increases, we consider Model 2* to be a better model of primes.

We note that the correction terms in   Models 2  and  2* {\it interfere} with one significant prediction of Model 1: eventually, the correction term step functions  
are strictly increasing as $n$ increases. The strictly increasing correction term prevents  the  prime gaps in Models 2 and 2* from getting too small as $n\to \infty$.\footnote{The minimum size of the $n$th prime gap should grow  roughly like $\log \log n$, under Simplifying Assumption 2.} {\it Models 2 and 2* do not model twin primes at large numbers.}  


\subsection{Critique of Models 2 and   2*}  We record concerns that might be addressed in the future. 

While Models 2 and 2* remain    in line with the main term of the prime number theorem, the missing contribution of partitions (and   factorizations) of lengths greater than three   creeps up, causing the models' estimates to become farther from exact as $n$ increases. Further corrections, either  refinements of our   assumptions or new ideas entirely, might produce models of higher accuracy.


Recalling that Simplifying Assumption 2 represents an expected overestimate of the prime gap length,  then we expect Model 2  produces an   {\it underestimate} for $\pi(n)$; Table \ref{table4}  supports this as $n$ increases.  Furthermore, the approximation of  $\pi_2(p_{i-1})$  by Landau's asymptotic  \eqref{landau}  in Model 2* introduces     inaccuracies at  values $i\geq 3$  small enough that the asymptotic does not provide a good approximation, and  appears to lead to an {\it overestimate} for $\pi(n)$ in Model 2* by inspection of Table \ref{table4}.   Even so, predictions from   Model 2* are generally better than those from Model 2;  this appears to be   a ``happy accident'' resulting from the first-order approximation \eqref{landau}.  Tenenbaum   \cite{Tenenbaum1}  provides an asymptotic series expansion for $\pi_k(n)$; the $k=2$ case   gives 
\begin{equation}\label{Ten}
    \pi_2(n)\  \sim\  \sum_{j=1}^{\infty}(j-1)!\  \frac{n \log\log n}{(\log n)^j} \  +\  \sum_{j=1}^{\infty}C_{j-1}\frac{n}{(\log n)^{j}}
\end{equation}
as $n\to \infty$, where the $C_i$ are real  constants; it is proved in \cite{MM} that $C_0=0.2614\dots$ is the  {\it Meissel--Mertens constant}. 
Substituting $n \mapsto p_{n}\sim n\log n$,    the $j=1$ summands in \eqref{Ten} give  
\begin{equation}\label{underestimate}
    \pi_2(p_{n})-n\log\log n\  \sim\    C_0 n.
\end{equation}
Thus $(i-1)\log \log (i-1)$ is   an {\it underestimate} for $\pi_2(p_{i-1})$, evidently leading to an {overestimate} for $\pi(n)$ in Model 2*. Indeed, by inspection of  the limited data in Table \ref{table4} and consideration  of  these differing estimates for $\pi(n)$, one can  conjecture   bounds on $\pi(n)$ as $n$ increases from Models 1, 2, and 2* in the  cases shown.  

\begin{conjecture}\label{conj}
Let ${\pi}^1(n)$ denote the estimate for $\pi(n)$ derived from Model 1,   let ${\pi}^2(n)$ denote the estimate  derived from Model 2, and let ${\pi}^{*}(n)$ denote the estimate derived from Model 2*. 
Then as $n\to\infty$ we have 
$${\pi}^2(n)\  \leq\  \pi(n)\  \leq\   {\pi}^{*}(n)\  \leq\   {\pi}^{1}(n).$$
\end{conjecture}



Simplifying Assumptions   1 and 2 suggest ${\pi}^2(n) \leq \pi(n)  \leq  {\pi}^{1}(n)$. As $n$ increases, we know that  $\pi^1(n)$       exceeds ${\pi}^{2}(n)$ and ${\pi}^{*}(n)$ due to their   correction terms, $\pi^2(n)\sim\pi^*(n)$ by \eqref{landau}, and ${\pi}^2(n) \leq {\pi}^{*}(n)$ follows from   \eqref{underestimate}; we cannot, however, see a reason why $\pi^*(n)$ should always be greater than $\pi(n)$ while ${\pi}^2(n) \leq \pi(n)$.  Table \ref{table4} also suggests $n/\log n \leq \pi(n)\leq \operatorname{li}(n)$ as $n\to\infty$, but Littlewood proves  $\operatorname{li}(n)-\pi(n)$    changes sign infinitely often \cite{Littlewood}, so we do not put too much faith in this empirical conjecture.

As we noted in Section \ref{model2pred}, Models 2 and 2* do not model twin primes at large numbers. 
To better model small prime gaps,  perhaps one can insert a factor $\phi(n)\in [0, 1]$ in the correction terms, e.g.  $\left\lfloor \phi(n)\cdot\left(\pi_2(p_{n-1}) - 2\gamma (n-1)\right) \right\rfloor$, such that $\phi(n)$ is closer to zero when $d(n-1)$ is smaller and is closer to one when $d(n-1)$ is greater; for example,   $\phi(n)=1-2/d(n-1)$. 
But this is not assured to produce a better model of  prime numbers and may adversely impact   estimates for $\pi(n)$. 


Finally, we make an aesthetic critique. Model 2 and Model 2*  produce sequences of integers that {fail} to imitate the primes in a  key way: the updated  models predict  {\it arbitrary primes are sometimes   even} as the correction terms 
can be odd or even. 
Noting that $2$ is the only even prime, we suggest a minor refinement of Models 2 and 2*:  replacing $\lfloor x \rfloor$ with a variant $\lfloor x \rfloor^*$ such that $\lfloor x \rfloor^*=\lfloor x \rfloor$ if $\lfloor x \rfloor$ is even, and $\lfloor x \rfloor^*=\lfloor x \rfloor + 1$ if $\lfloor x \rfloor$ is odd. This  ensures all     modeled prime gaps are even and all  predicted prime numbers are odd, and should not significantly affect estimates. 

\section{Concluding remarks}

The authors posit Model 1 as a conceptual model that attempts to explain aspects of prime distribution based on elementary assumptions, which we derive from observations about  the  partition norm and supernorm. The model not only is compatible with major observations such as the prime number theorem and   twin prime conjecture, but also predicts   local fluctuations in prime gaps that the prime number theorem does not. We treat these   predictions about  fluctuations as   representing probabilistic statements and give them a cursory check, with mixed findings. 
By contrast, Model 2 and Model 2* are proposed as computational models that  introduce simple  corrections to Model 1 to yield estimates for $\pi(n)$ that are surprisingly good. 

Because it is compatible with     observations of primes and suggestive   of  new phenomena that can be tested, we hope     the approach we present  provides a useful complement  to  probabilistic models.

  \begin{figure}
\includegraphics[scale=.20]{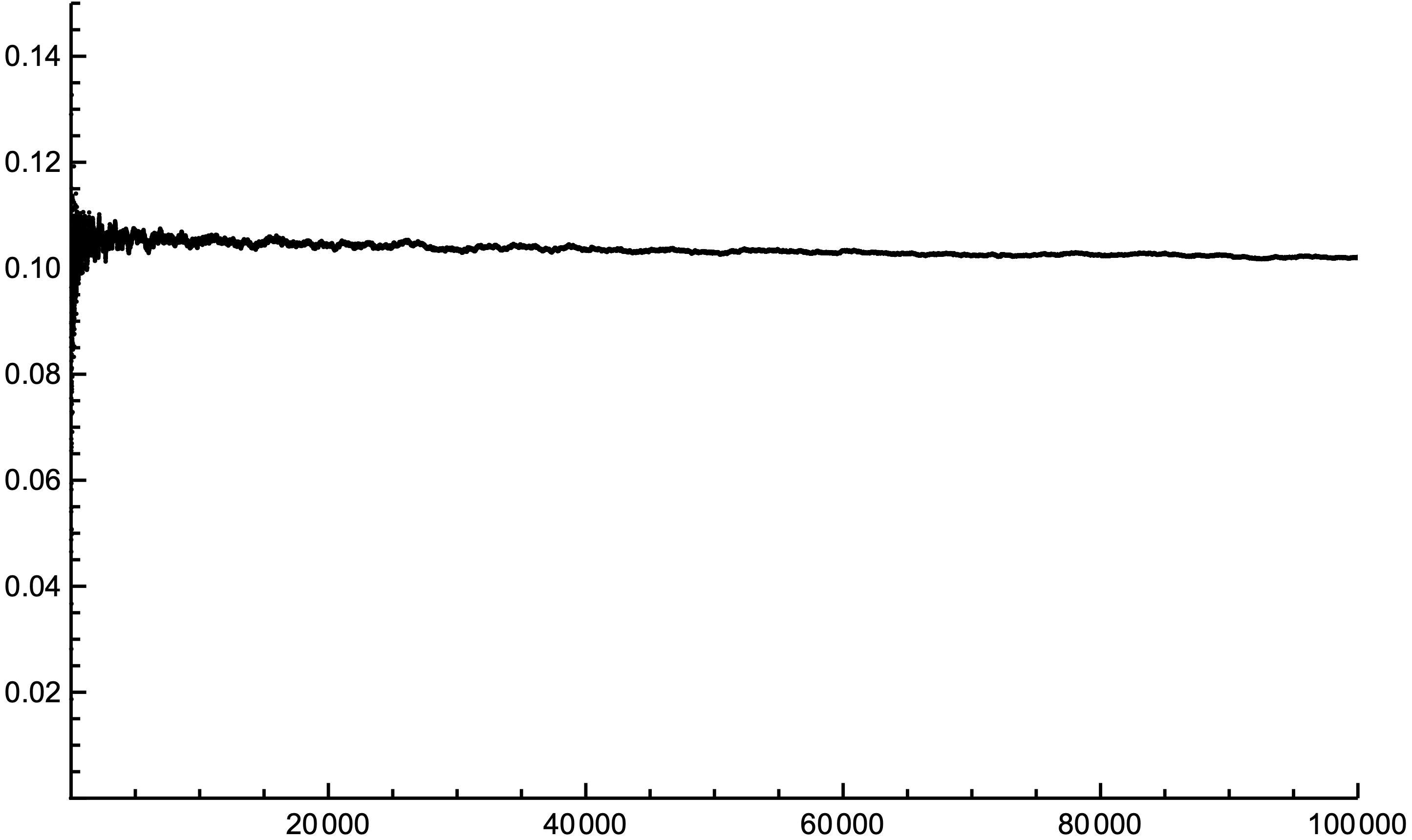}
 \caption{Relative error in the Model 1 estimate for $p_n$ up to $n=100{,}000$}
 \label{fig4}
 \vskip.15in \end{figure}

  \begin{figure}
\includegraphics[scale=.2]{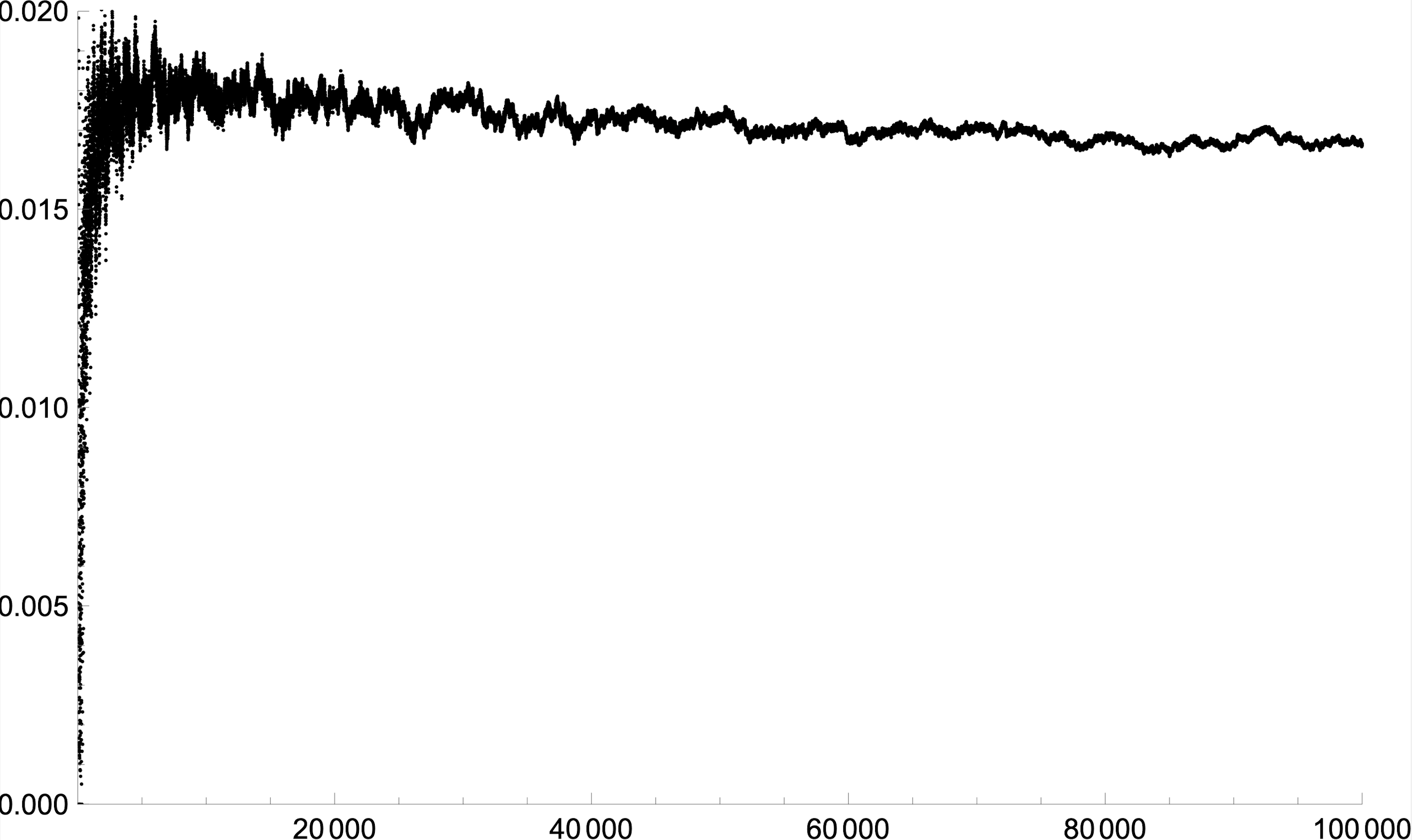}
 \caption{Relative error in the Model 2 estimate for $p_n$ up to $n=100{,}000$}
 \label{fig5}\vskip.15in
  \end{figure}

  \begin{figure}
\includegraphics[scale=.211]{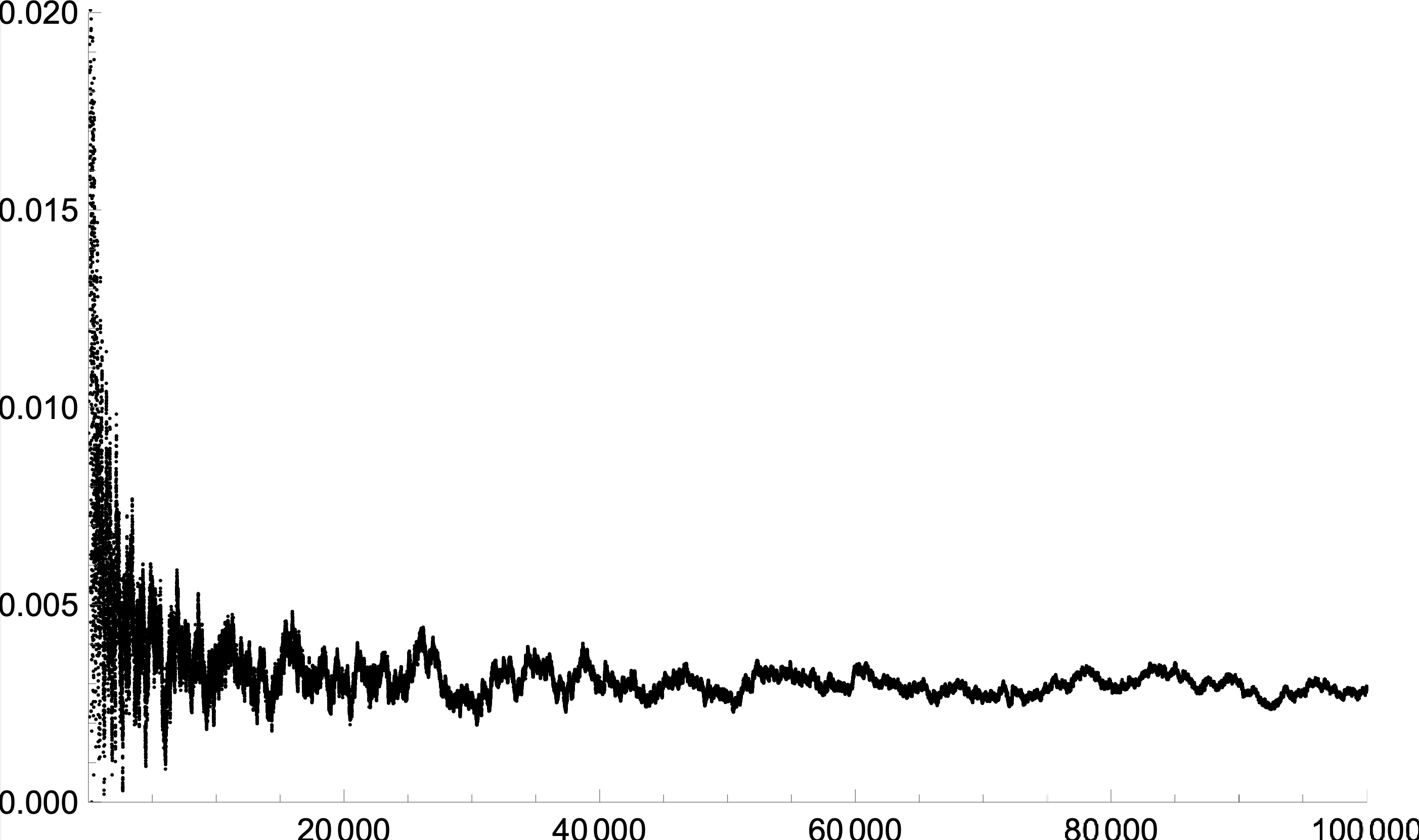}
 \caption{Relative error in the Model 2* estimate for $p_n$ up to $n=100{,}000$}
 \label{fig6}
  \end{figure}


\section*{Acknowledgments}

For calculations we use Wolfram Mathematica
computer algebra system and Python programming language. The authors are  very  thankful to Eli DeWitt     and Alexander Walker   for computer programming  that advanced our project, and to Maxwell Schneider for programming consultation. Furthermore, we are    grateful to  Krishnaswami Alladi, George E. Andrews,  William Craig, Philip Cuthbertson, Daniel Desena, Michael Filaseta, Andrew Granville,  Timothy Havens, Abhimanyu Kumar, Jeffrey Lagarias, David Leep, Jim McIntyre,  Paul Pollack, Ken Ono, Meenakshi Rana,  James Sellers,  Steven Shawcross, Andrew V. Sills and Hunter Waldron for conversations that informed this paper,  and to Ivan Ramirez Zuniga for advice about scientific modeling. In particular, we thank Jeffrey Lagarias and Paul Pollack for offering well-placed suggestions with respect to an earlier draft of this work, Michael Filaseta for invaluable notes on statistical interpretations of our data, and Eli DeWitt for checking our numerical computations prior to publication.    We also thank  the anonymous referee, who made     useful suggestions.




\end{document}